\documentclass[12pt]{amsart} 
\usepackage{enumerate}
\usepackage{amsmath}
\usepackage{color}
\usepackage{amssymb}
\usepackage{braket}

\numberwithin{equation}{section}

\theoremstyle{plain}
\newtheorem{thm}{Theorem}

\newtheorem{conj}{Conjecture}
\newtheorem{cor}{Corollary}
\newtheorem*{comm}{Comment} 
\theoremstyle{definition}

\newcommand\md{\textrm{\upshape mod }}

\newcommand{\ph}{\varphi}

\newcommand{\Oh}{\mathcal{O}}

\newcommand{\Z}{\mathbb{Z}}
\newcommand{\C}{\mathbb{C}}
\newcommand{\Q}{\mathbb{Q}}
\newcommand{\K}{\mathbb{K}}
\newcommand{\setE}{\mathcal{E}}
\newcommand{\e}{\mathrm{e}}
\newcommand{\li}{\operatorname{li}}
\newcommand{\Eexpl}{E_{\mathrm{expl}}}
\newcommand{\Ss}{\mathcal{S}}

\allowdisplaybreaks
\usepackage{blindtext}
\usepackage{comment}

\begin{document}
\title[ Conditional Bounds on Siegel Zeros ]
{Conditional Bounds on Siegel Zeros } 
\author{Gautami Bhowmik}
\address{G. Bhowmik: Laboratoire Paul Painlev{\'e}, Labex-CEMPI, Universit{\'e} de Lille, 59655
Villeneuve d'Ascq Cedex, France}
\email{gautami.bhowmik@univ-lille.fr}

\author{Karin Halupczok}
\address{K. Halupczok: Mathematisches Institut der
Heinrich-Heine-Universit{\"a}t D{\"u}sseldorf, Universit{\"a}tsstr.~1,
40225 D{\"u}sseldorf, Germany}
\email{karin.halupczok@uni-duesseldorf.de}

\thanks{We thank Andrew Granville and Lasse Grimmelt for helpful comments and the referee for
   an improved presentation.}
\date{}
\maketitle
\hfill {\it Dedicated to Melvyn Nathanson.}
\begin{abstract}
We present an overview of bounds on zeros of $L$-functions and
obtain some 
improvements under weak conjectures
related to the Goldbach problem.
\end{abstract}
\maketitle

\section{Introduction }\label{Intro}  
The  existence of  non-trivial real zeros of a Dirichlet $L$-function
would contradict the Generalised Riemann Hypothesis.  One possible counter-example, called the Landau--Siegel zero,
 is  real and simple
and  the region in which it could eventually exist 
is important to determine.
In 1936 Siegel gave a quantitative estimate on the distance of  an exceptional zero 
from the line $\Re s=1$.
The splitting into cases depending on whether such an
exceptional zero exists or not happens to be an important technique often used in analytic number theory,
for example in the theorem of Linnik.
In the first section we discuss properties of the Siegel zero and 
results assuming classical and more recent hypotheses. This part of the paper is expository.

In the second section we present a conditional bound. In 2016 Fei  improved Siegel's bound for certain moduli
under a weakened Hardy--Littlewood  conjecture  on the Goldbach problem of representing 
an even number as the sum of two primes.  In Theorem~\ref{thm:bettersiegel} 
and Corollary~\ref{cor:cor1}
we further weaken this conjecture and 
enlarge the
set of moduli 
to include more Dirichlet characters.

\section{Background} 
Consider a completely multiplicative, periodic  arithmetic function $\chi : \Z \rightarrow \C$ where for $q \ge  1$ there exists a group homomorphism
$\tilde\chi : (\Z/q\Z)^{\times} \rightarrow \C^{\times}$
such that $\chi(n) = \tilde\chi (n(\md q))$ for $n$ coprime to $q$ and $\chi(n) = 0$ if not.
We call  $\chi$ a Dirichlet character $(\md q)$. In fact, if $(n,q)=1, \chi(n)$ is a $\phi(q)^{th}$ complex root of unity.
We denote the
{\it principal} character $\md q$, whose value 
$\chi(n)$ 
is always 1 
for $n$ coprime to $q$, 
by $\chi_0\ (\md q)$. 
The {\it order} of  $\chi$ 
is the least positive integer $n$ such that $\chi^n=\chi_0$, both characters having the same modulus.
A non-principal character is called {\it quadratic} if  $\chi^{2}=\chi_{0}$.
In the case where $\chi$ 
always takes a real value, the possibilities being only 0 or $\pm 1$, it is called a {\it real} character, 
otherwise it is called {\it complex}.
A character modulo $q$ is termed {\it primitive}  and $q$ its {\it conductor} if it cannot be factored as ${\chi = \chi' \chi_0},$ where ${\chi_0}$ is a principal character and ${\chi'}$  a character of modulus strictly less than $q$.
For a given
$\chi\ (\md q)$, there is a unique primitive character $\tilde\chi (\md \tilde{q})$ with least possible $\tilde{q}$, where $\tilde{q}\mid q$,
that {\it induces} $\chi$,
such that $\chi$ and $\tilde\chi$ have the same value at all $n$ coprime to $q$. 

The $L$ series were introduced in 1837 by Dirichlet who used them to prove an analytic formula for the class number 
and  the infinitude of primes in any  arithmetic progression.
For $s=\sigma +it$, $\sigma >1$, and a Dirichlet character $\chi$, we consider the Dirichlet $L$-function
\[
  L(s,\chi) =  \sum_{n=1}^\infty
\frac{\chi(n)}{n^s}=\prod_{p\  \text {prime}}\Big(1-\frac{\chi(p)}{p^s}\Big)^{-1}.
\] 
Note that since $L(s, \chi) = L(s,\tilde\chi) \prod_{p\mid q}  (1 -\tilde\chi(p)p^{-s} )$ 
there is no loss in considering only primitive characters for obtaining analytic properties.

Let $\rho_{\chi}=\beta_\chi +i\gamma_\chi$ be the non-trivial
zeros of the $L$-function.  These are known to be contained in the strip $0<\Re(s)<1$ though
according to the Generalised Riemann Hypothesis (GRH),  the only possible value of $\beta_\chi$ is $1/2$.
While the GRH
remains out of reach, much unconditional work has been directed towards finding zero-free regions  for $L$-functions.

Around hundred years ago it was proved
that real zeros close to $\Re s=1$ are indeed rare. More precisely,

\begin{thm}[Landau--Page]  \label{thm:page}
There is an absolute constant $c > 0$ such that for any $Q,T\ge 2$,
the product  $ \prod_{q\le Q} \prod_{\chi \mod q} ^{\star}L(s,\chi)$ has at most one zero of an $L$-function in the region
\[
   |t|\leq T,\ 1-\sigma\le \frac{c}{\log (QT)};
\]
where $^\star$ runs over all primitive real characters of modulus $q$. If such a zero exists, 
then it is real and associated to a unique, quadratic $\chi \mod q$.
\end{thm} 
This eventual `bad' zero  contradicting the GRH  is called  the exceptional or Siegel or
{\it Landau--Siegel zero} and the corresponding character is called the
{\it exceptional} character.  We denote the Landau--Siegel zero by $\beta_{\chi}$ or
simply $\beta$. 

The work of Landau and Siegel  provide bounds on  the proximity of such a zero on the real axis from $s=1$.
Quantitatively,

\begin{thm}[Siegel]\label{thm:Siegel35}
For an exceptional zero $\beta$ associated to a primitive character $\chi$ of conductor $q$ and any $\epsilon>0$ there is a constant $c(\epsilon) > 0$ such that 

\begin{equation}
  1-\beta \geq \frac{c(\epsilon)}{q^{\epsilon}}.
\end{equation}
\end{thm} 

Unfortunately, the constant $c(\epsilon)$ cannot be computed effectively
for any $\epsilon < 1/2$,
which is a serious difficulty for many applications.
In 1951,   Tatuzawa \cite{Tat} did improve on Siegel's theorem to give  an effective version for almost all cases by proving that
for any positive $\epsilon $
there does exist an effectively computable positive constant
$c(\epsilon)$ such that for all quadratic characters $\chi$, with  at
most one exception, $L(s, \chi)$ has no zeros in the interval $[1 -c(\epsilon)/q^{\epsilon}, 1].$

\subsection{Repulsion Property}
The possible exceptional zero would force all other zeros, real or otherwise, of all $L$-functions of the same modulus away 
from the real axis. We  state a quantitative version of the Deuring--Heilbronn result of 1933-34.

\begin{thm}
 There exist effective constants $c, c' > 0$ such that for any $T \ge  2$ and any $q \ge  1$, if for some quadratic $\chi( \md q), L(s,\chi) $ has an exceptional zero
$\beta \in [1-c/\log (qT),1]$,
then $\prod_{\chi } L(s,\chi)$,  the product over all characters of
 modulus $q$ including the exceptional one, has no other zero in the domain
\[\sigma \ge 1-\frac {c' |\log((1-\beta)\log qT) |}{\log qT},\  t \le T.\]
\end{thm}

Here is a reformulated version of the repulsion phenomenon also due to Linnik in 1944
which he used to find the size of the least prime in an arithmetic progression.

 \begin{thm}
\label{thm:rep_linnik}
If  there exists an exceptional zero $\beta$
with $ 1 - \beta=\frac{\varepsilon}{\log q}$ for $\varepsilon$ sufficiently small,  then all other
zeroes $\sigma+it$ of $L$-functions of modulus $q$ are such that
\[  
1  - \sigma \geq c \frac{\log \frac{1}{\varepsilon}}{\log(q(2+|t|))}
\]
for an absolute positive constant $c$.
\end{thm}

Compared to the classical estimate  $ \sigma \ge 1 - \frac{c}{\log(q(2+|t|))} $ with some absolute positive constant 
$c$, 
for a region where $L(s, \chi)$, for any $\chi\ \md q$, 
contains no zeros except at most one eventual exception, we now have a zero-free region wider
 by a factor of $\log \frac{1}{\varepsilon}$.

Theorem \ref{thm:rep_linnik}  above was  strengthened by Bombieri \cite{Bom} to

\begin{thm} \label{thm:rep_Bom}
Let $T\geq 2$ and 
$\beta$ be an exceptional zero with respect to the otherwise zero-free region  $\sigma \ge \frac{c}{\log T}, |t| \le T$, then there exist constants $c_1, c_2$ such that if
$(1-\beta)\log T\le c_2/e$, then for any zero $\sigma+it\neq \beta$ of
$L(s,\chi)$, we have  
\[1-\sigma \ge {c_1}\frac{\log\frac{c_2} {(1-\beta)\log T}}{\log T},\] 
where $ |t|\le T$ for every primitive $\chi$ of modulus $q\le T$.
\end{thm}

This can be written in terms of a density estimate.  Let $N(\alpha, q,T)$  denote the number of zeroes, counted with multiplicity, of any $ L$  function of modulus $q$ with $\alpha \leq \sigma \leq 1$  and  
$0 \leq t \leq T$  and let $N'$ denote the case when  $\beta$ is omitted.  Then there is an  improvement

 \[  N'(\alpha,q,T) \ll (1-\beta)(\log q) 
\Big(1 + \frac{\log T}{\log q}\Big) (qT)^{O(1-\alpha)} 
 \]
with effective implied constants  over Linnik's density estimate

\[  N(\alpha,q,T) \ll (1-\beta)(\log qT) (qT)^{O(1-\alpha)} .
 \]

In \cite{FI} Friedlander and Iwaniec state the `ultimate Deuring--Heilbronn property' as
\begin{thm}\label{thm:rep_FI}
Let $\chi(\md q)$ be a real primitive character of conductor $q$ with the largest real zero $\beta$ and let 
$\eta=\frac{1}{(1-\beta)\log q}\ge 3$. 
Then $L(s,\chi)$
has no zeros other than $\beta$ in the region 
$\sigma > 1-\frac{c\log \eta}{\log q(|t|+1)}$ where $c$ is an absolute positive constant.

\end{thm}

\subsection {Bounds for $L(1,\chi)$}

We know  at least since Hecke and Landau that  zeros of $L(s,\chi)$ and its value at $s=1$ are closely related.
If $L(1,\chi)$ is sufficiently small relative to the conductor, then there  is a Siegel zero and conversely.
More precisely, if $L(1,\chi)\leq \frac{c}{\log q}$ for a small constant $c>0$,
then $1-\beta \leq \frac{1}{\log q}$.
Using the Deuring--Heilbronn repulsion property,
Friedlander and Iwaniec recently proved  \cite{FI}
that 
\begin{equation}\label{eqn:FI}
\{1-\beta \ll (\log q)^{-3}\log\log q \}\implies  \{L(1, \chi) \ll (\log q)^{-1}\}.
\end{equation}
Goldfeld \cite{Gol1} 
provided an asymptotic result for the location of the Siegel zero. In fact, when $1-\beta < \frac c{\log q}$, he obtained a precise asymptotic formula 
 \begin{equation}\label{eq:Goldf}
 1-\beta \sim \frac{6}{\pi^2}L(1,\chi)\Big(\sum a^{-1}\Big)^{-1}
 \end{equation}
 where the summation is over all reduced quadratic forms $(a,b,c)$ of discriminant $-q$.  

Dirichlet  expressed the value of $L(1, \chi)$ where $\chi(n)=(\frac{-q}n) $ is the real primitive character of conductor $q$ in terms of the number $h(q)$ of equivalence classes of binary quadratic forms of discriminant  $q$,
which can equivalently be formulated in terms of  the number of ideal classes of an imaginary quadratic number field $\K=\Q(\sqrt -q)$.  
From the class number formula $L(1,\chi) =\frac{\pi h(q)}{ \sqrt {-q}} $ (here $q<-4$)
 one obtains the non-vanishing of 
$L(1, \chi)$ which is not that obvious when $\chi$ is real, and leads to the prime number theorem in arithmetic progression.
Another obvious consequence of the above formula is the elementary lower bound 
\[L(1, \chi) \gg \frac{1}{q^{1/2}}.\]
Bounding $L(1,\chi)$  is  equivalent to estimating the size of the class number of the imaginary quadratic field $\K=\Q(\sqrt -q)$, another important question in number theory. The corresponding formula is $ L(1, \chi)=\frac{2\pi h_\K}{w_{K}\sqrt {d_{\K / \Q}}}$ for the class number $h_\K $  and discriminant $d_{\K / \Q}$ of $\K$, $w_{K}$ the order of the group of units with regulator being 1 and the LHS being the residue of the Dedekind zeta function at $s=1$. 

Good effective lower bounds are more difficult to obtain.  Goldfeld \cite{Gol}  in 1976  using known cases of the Birch and Swinnerton-Dyer 
 conjecture for elliptic curves showed that
\[L(1,\chi)\gg \frac{\log q}{\sqrt q(\log \log q)} \]
for $q\ge 3$, the implied constant being effective. 
This together with  Gross--Zagier's work of 1983  is a major step in the Gau\ss\  class number problem. 
\begin{thm} [Goldfeld--Gross--Zagier] For every $\epsilon > 0$ there exists an effectively computable positive constant $c$
such that $h(-q) >(c\log q)^{1-\epsilon}$.
\end{thm}
 This corresponds to a zero-free region of $L(s,\chi)$ of size $[1-c_0\frac {\log^{c_1}(q)}  {\sqrt q}, 1]$ for some effective positive constants $c_0, c_1$ and  for all real primitive characters.

Oesterl\'e's calculation of the involved constant in 1985 makes it possible to state this bound
for $q>0$ as 
\[L(1,\chi)>\frac{\pi}{55\sqrt q}\log q \prod_{p\mid q}  \Big(1- \frac{2\sqrt p}{p+1}\Big).\]

Rather recently Bennett et al.~\cite{BMOR}  proved that if $\chi$ is a primitive quadratic character with conductor $q > 6677$, 
then $L(1, \chi) > \frac{12}{\sqrt q} $. 

We are still far from the plausible
lower bound  $L(1, \chi) \gg (\log q)^{-1}$ which holds in many cases, for example for complex characters
with an effective constant.

Aisleitner et al.\ \cite{AMMP}  in 2019 showed the existence, for $q$ sufficiently large, of an extremal non-principal character which satisfies,
for constants $C$ and $\gamma$,
$ |L(1,\chi)|\ge e^{\gamma}(\log\log q+\log\log\log q-C)$ 
 using the method of resonance for detecting large values of the Riemann zeta function. Up to the constant, this corresponds to the predicted order of the extremal values.

A simple unconditional upper bound  is $L(1, \chi) \ll \log q$.
The implied constants have been worked on by a variety of methods.
For example, for complex characters Granville and Soundararajan \cite{GS} determine the  constant 
$c_k$  for primitive characters of order $k$ for which the bounds $|L(1, \chi)| \leq (c_k+o(1))\log q$ hold true. 
For real primitive characters, the constant 
$c_2=\frac{1}{4}(2-\frac {2}{\sqrt e}+o(1))\log q$ was obtained by Stephens for prime characters \cite{Ste} and Pintz \cite{Pin} extended this to non-prime characters. 

\subsubsection{Conditional Bounds} The optimal bounds of $L(1,\chi)$ under the condition
of GRH  are
\[ (\log \log q)^{-1}\ll L(1,\chi)\ll \log \log q
\]
where the implied constants are effective. Precisely speaking,
\begin{thm}[Littlewood 1928]
If the Generalised Riemann Hypothesis is true then
\[\Big(\frac 1{2}+o(1)\Big)\frac{\pi^2}{6e^{\gamma}\log\log q}\le L(1,\chi)\le(2+o(1))e^{\gamma} \log \log q
\] where $\gamma$ is Euler's constant.
\end{thm}
Only the implied constants in the above can be improved because there actually exist infinitely many $q$ for which the special value 
of the corresponding character at $s=1$ correspond to the above magnitude of orders. 
The classical unconditional  $\Omega$ results, that
$( 1+o(1))\frac{\pi^2}{6e^{\gamma}\log\log q}\ge L(1,\chi)$
and 
$L(1,\chi)\ge(1+o(1))e^{\gamma} \log \log q$
hold for infinitely many $q$ \cite{Chowla} show that
 there is a factor of 2 that remains undetermined for the extreme values.

We cite  one example of a recent refined upper and lower bound established by  Lamzouri et al.~\cite{LLS}  assuming the GRH  for characters of large conductor and studying certain character sums.  For  $q\ge10^{10}$, the bounds obtained therein can be written in a simplified  manner as
\[\frac{\pi^2}{12e^{\gamma}\log\log q}< |L(1,\chi)|<2e^{\gamma}\log\log q.\]

The lower bound can be improved a lot by admitting the existence of Landau--Siegel zeros,
and thus weakening the GRH. 
One such assumption, sometimes called the Modified Generalised Riemann Hypothesis (MGRH), is that 
 all the zeros of  $L(s, \chi)$ lie either on the critical line  or on the  real axis. Sarnak and Zaharescu \cite{SZ}  
showed that if all Dirichlet $L(s,\chi)$ with $\chi$ real  satisfy the MGRH then 
\[ L(1, \chi)\ge \frac{c^{\epsilon}}{(\log |q|)^{\epsilon}}\]
for any positive $\epsilon$. The above constant is ineffective but the bounds can be made effective under certain additional conditions. 
These bounds use the explicit formula with an appropriately constructed  kernel function.

 Assuming that the GRH 
holds except for one possible exception, Friedlander and Iwaniec obtained an improved version of  (\ref{eqn:FI}). 
They proved that:
\begin{thm}[\cite{FI}]
\label{thm:FI} 
Let the GRH be true except for only one $\beta>3/4$.
Now if  $1-\beta\ll (\log \log q)^{-1}$, then 
\[ 1-\beta\ll L(1,\chi)\ll (1-\beta)(\log \log q)^{2}.
\]
\end{thm}

\section{Better Siegel zero bounds from Weak 
Goldbach  conjectures} 

Connections between Siegel zeros and the Goldbach problem
were studied, for example in
\cite{BHMS} and \cite{Fei}. Among the classical conjectures of the Goldbach problem is one due to 
Hardy and Littlewood in 1923
that predicts an equivalence between  the number of  representations of an even number  as a sum of two primes and a singular series, 
$g(n)=~\sum_{n=p_1+p_2}1 \sim \Ss(n)$,
where 
  \[
  \Ss(n):= \frac{n}{\varphi(n)}\prod_{p\nmid n}\Big(1-\frac{1}{(p-1)^2}\Big)\cdot \frac{n}{\log^{2}n}= 2C_{2} \Big(\prod_{\substack{p\mid n \\ p>2}} \frac{p-1}{p-2}\Big) 
\cdot \frac{n}{\log^{2}n} 
  \]
with the twin prime constant
\[C_2=  \prod_{p>2} \Big(1-\frac{1}{(p-1)^{2}}\Big)\]
which is approximately 0.66. 

Fei \cite{Fei} obtained an upper 
bound for $\beta$ under a weakened 
form of the  Hardy--Littlewood conjecture (WHL), namely
\begin{conj}[WHL]\label{conj:WHL}
There exists a positive contant $\delta$ such that $g(n)\ge
\frac{\delta n}{\log^2 n}$  for every even integer $n>2$.
\end{conj}

Considering the size of $\Ss(n)$ we could even expect a $\delta>1.32$,
but in this weakened form of the Hardy--Littlewood conjecture, we assume 
only the existence of some small
positive $\delta$.

We now state Fei's theorem.
\begin{thm}[\cite{Fei}] \label{thm:Fei}
If the WHL-conjecture is true and if there is
an exceptional zero $\beta$ 
for a character $\chi$ with a prime modulus $q\equiv 3\mod 4$, then there
exists a positive constant $c$ such that  $1-\beta\ge \frac{c}{\log^2
q}$.
\end{thm} 

Here, the corresponding region for the exceptional zero $\beta$ 
is meant to be that of Theorem~\ref{thm:page} with $T=q$. We will keep to
this convention for the rest of this section.

One would like to know if it is possible  to include other moduli in the above result or to relax the assumed WHL-conjecture.

Here we generalise Fei's result in these two aspects
and obtain a conditional improvement of
Siegel's bound (Theorem~\ref{thm:Siegel35}) for certain exceptional
characters  which includes Fei's modulus
condition (Corollary~\ref{cor:cor1}).
Our result still assumes the weak Hardy--Littlewood conjecture but allows certain exceptions (WHLE) making it  weaker than the WHL.
Our proof is similar to that of Fei's but
 exploits, apart from the use of the WHLE,
the generalisation to suitable composite moduli $q$.

\begin{conj}[WHLE]\label{Exc} 
Suppose that $x$ is sufficiently large, and $q\leq x/4$. Then
we have, with  at most 
$x/8q$  exceptions,
\[g(n)
\gg\frac { n}{\log^{2}n}\] for  
the multiples $n$ of $q$ in the interval $x/2<n\leq x$.
\end{conj}

\begin{thm}\label{thm:bettersiegel}
Assume the WHLE Conjecture to be true. 
Let $q$ be a sufficiently large integer and 
$\chi$ be a primitive
character mod $q$ with $\chi(-1)=-1$ such that there is an
exceptional zero $\beta$ of $L(s,\chi)$.  Then there exists an effective
constant $c>0$  such that $1-\beta \geq\frac{c\ph(q)}{q\log^2(q)}$.
\end{thm}

\begin{proof}[Proof of Theorem~\ref{thm:bettersiegel}]

\underline{Step 1.}
 We prove the following \underline {lower bound} for the sum
\begin{equation}\label{boundS} S=\sum_{k=1}^q \Big(\sum_{2<p\le x}
\e\Big(\frac{kp}{q}\Big)\Big)^2\ge \frac{\delta x^2}{8\log^2 x}
\end{equation} for any sufficiently large real $x>2$ and some small $\delta>0$. 

For this, we first note note that
\[
\begin{aligned} S&=\sum_{k=1}^{q} \sum_{2<p_1, p_2\le x}
\e\Big(\frac{k(p_1+p_2)}{q}\Big) 
\\ &=\sum_{n\le 2x}\sum_{k=1}^{q}
\e\Big(\frac{kn}{q}\Big)\sum_{\substack{2<p_1, p_2\le x\\ p_1+p_2=n}}
1=\sum_{\substack{n\le 2x\\ n\equiv 0(q)}}q \sum_{\substack{2<p_1,
p_2\le x\\p_1+p_2=n}} 1.\\
\end{aligned}
\] 

Let $x$ be large enough with   
\begin{equation}\label{eq:qbound}
   q\le x/4. 
\end{equation} 
Hence under the assumption of  Conjecture~\ref{Exc},
for all even $n$ in the interval $x/2<n\leq x$
that are divisible by $q$, we have
\[
   \sum_{\substack{2<p_1,
p_2\le x\\p_1+p_2=n}} 1\geq\delta \frac{x}{\log^{2}x}
\]
for some constant $\delta>0$, with the possible exception of at most
$x/8q$
 such $n$.
Let $\setE$ be the set of these exceptions. 

Keeping this in mind, we obtain the lower bound
\begin{multline}\label{boundsum} 
S\ge q\sum_{\substack{x/2<n\le
x\\2|n\\n\equiv 0(q)}}\ \sum_{p_1+p_2=n}1\ge q \sum_{\substack {x/2<n\le
x\\2|n\\ n\equiv 0(q)\\ n\notin \setE}} \delta \frac{x}{\log^{2}x} \\
\geq \frac{q\delta  x}{\log^{2}x}\Big(
\frac{x}{4q} -\frac{x}{8q}\Big) = \frac{\delta x^{2}}{8\log^{2}x}
\end{multline}
as was to be shown.

\bigskip
\underline{Step 2.} We now prove a more \underline {explicit expression}
for the sum $S$ in the first step.

In \eqref{boundS}, the sum over $p$ is subdivided into parts
 depending on whether or not $q$ is divisible by $p$, i.e.\

\begin{multline*}
 S_1=\sum_{2<p\le x}
\e\Big(\frac{kp}{q}\Big)=\sum_{\substack{2<p\le x\\p\nmid q}}
\e\Big(\frac{kp}{q}\Big)+\sum_{\substack{2<p\le x\\p\mid q}}
 \e\Big(\frac{kp}{q}\Big)
=\sum_{\substack{2<p\le x\\(p,q)=1}} \e\Big(\frac{kp}{q}\Big)+\Oh(\omega(q))\\
=\sum_{\substack{1\le a\le q\\(a,q)=1}}
\e\Big(\frac{ka}{q}\Big)\sum_{\substack{2<p\le x\\p\equiv a(q)}}
1+\Oh(\omega(q))
=c_q(k)\sum_{\substack{2<p\le x\\p\equiv a(q)}} 1 + \Oh(\omega(q)),
\end{multline*}
where $c_q(k)$ denotes the above Ramanujan sum 
$$
\sum _{\substack{1\leq a\leq q \\ (a,q)=1}}e({\tfrac {ka}{q}})
$$ 
while  $\omega(q)=\#\{p\mid q\}$.
Using the prime number theorem  in arithmetic progressions ~\cite{MV},
the last sum over $p$ is written out as 
\begin{equation}
\label{eq:pntap}
\sum_{\substack{2<p\leq x \\ p\equiv a(q)}} 1=
\frac{\li(x)}{\ph (q)}-\frac{\chi(a)}{\ph (q)}
\int_2^x\frac{u^{\beta-1}}{\log u}du +\Oh(x\exp(-\tilde{c}\sqrt{\log
x}))
\end{equation}
for some constant $\tilde{c}>0$, which holds
uniformly in 
\begin{equation}
\label{eq:zwqbound}
q\leq \exp(C\sqrt{\log x})
\end{equation}
for any $C>0$ (this range
is consistent with \eqref{eq:qbound}, though we could have chosen a larger $x$).
Hence, this gives 
\begin{equation}
\label{MT}
 S_1=M+c_q(k)\frac{\li(x)}{\ph(q)} +\Oh\Big(
qx\exp(-\tilde{c}\sqrt{\log x})\Big)
\end{equation}
with 
term $M$ being
\[ M=\frac{-1}{\ph(q)}\sum_{a=1}^q
\e\Big(\frac{ak}{q}\Big)\chi(a)\int_2^x\frac{u^{\beta-1}}{\log u}du
=\frac{-\tau_k(\chi)}{\ph(q)}\int_2^x\frac{u^{\beta-1}}{\log u}du,
\]
since $\chi(a)=0$ if $(a,q)>1$, with the Gau\ss\ sum 
\[
\tau_k(\chi)=\sum_{a=1}^q \e\Big(\frac{ak}{q}\Big)\chi(a).
\]
Inserting the expansion $\int_{2}^{x}\frac{u^{\beta-1}}{\log u}
du=\frac{x^{\beta}}{\beta \log x}+\Oh(\frac{x^{\beta}}{\log^{2}x})$
yields
\[
M=\frac{-\tau_k(\chi)}{\ph(q)}\cdot\frac{x^\beta}{\beta \log x} +
\Oh\Big(\frac{q^{1/2}}{\ph(q)}\cdot\frac{x^\beta}{\log^2x}\Big)
\]
from the estimate  $\tau_{k}(\chi)\ll q^{1/2}$.

We substitute this expression in \eqref{MT}
and the resulting approximation for $S_1$ into 
$S$ to get
\begin{multline}
\label{gauss}
  S=\sum_{k=1}^q S_1^2  = \sum_{k=1}^q \Big(
  c_q(k)\frac{\li(x)}{\ph(q)}
  - \frac{\tau_k(\chi)}{\ph(q)}\cdot\frac{x^\beta}{\beta \log x}  \\
\hfill + \Oh\Big(qx\exp(-\tilde{c}\sqrt{\log x}) 
+ \frac{q^{1/2}}{\ph(q)}\cdot\frac{x^\beta}{\log^2x}\Big)  \Big)^2 \\
=\sum_{k=1}^q \Big(  c^2_q(k)\frac{\li^2(x)}{\ph^2(q)}
+\frac{\tau^2_k(\chi)}{\ph^2(q)}\cdot\frac{x^{2\beta}}{\beta^2 \log^2 x} \Big)\\
+ \Oh\Big(\frac{q^{1/2}x^\beta}{\ph(q)\log x}q^2x\exp(-\tilde{c}\sqrt{\log x})
+\frac{q^2x^{2\beta}}{\beta\ph^2(q)\log^3x} 
+q^3x^2\exp(-2\tilde{c}\sqrt{\log x})\Big),
\end{multline}
where all the mixed terms containing the Ramanujan sum have now disappeared
since
\begin{multline*}
   \sum_{k=1}^q  c_q(k)\tau_k(\chi)=\sum_{k=1}^q 
   \sum_{\substack{1\leq a\leq q\\(a,q)=1}}\sum_{\substack{1\leq b\leq q\\(b,q)=1}}
    \e\Big(\frac{ak}{q}\Big)\e\Big(\frac{bk}{q}\Big)\chi(b) \\
    = \sum_{\substack{1\leq a\leq q\\(a,q)=1}}
    \sum_{\substack{1\leq b\leq q\\(b,q)=1\\b\equiv -a(q)}} \chi(b) q =
    q\sum_{\substack{1\leq a\leq q\\(a,q)=1}} \chi(-a) =0
\end{multline*}
and $\sum_{k=1}^q c_q(k)=0$. The $\Oh$-term in \eqref{gauss} simplifies to
\begin{equation}\label{eexpl}
  E_{expl}=\Oh\Big(\frac{x^{2}}{\log^3x}\frac{q^2}{\ph(q)^2} 
+q^3x^2\exp(-\tilde{c}\sqrt{\log x})\Big).
\end{equation}

For the main term in \eqref{gauss}, we
use properties of Gau\ss\ sums (\cite[p.287]{MV}),
\[
\tau_k(\chi)=\Big\{
\begin{aligned} \bar{\chi}(k)\tau_1(\chi),&\ (k,q)=1\\ 0,&\
\text{ else},\\
\end{aligned}
\]
so that the sum over $\tau_k^2(\chi)$ in \eqref{gauss} becomes
\[
\frac{1}{\ph^2(q)} \sum_{\substack{k=1\\ (k,q)=1}}
^{q-1}\tau_1^2(\chi)\bar{\chi}^{2} (k)=\frac{q}{\ph(q)}\chi(-1)
\]
since $\tau^{2}_{1}(\chi)=\chi(-1)q$ and $\chi^{2}=\chi_{0}$.
Similarly, we have
\[ \sum_{k=1}^q c_q^2(k)=q\ph(q),
\]
see \cite[p.113]{MV}. Hence
\begin{equation}\label{explicit} S= \frac{q}{\ph(q)}\li^2(x)
  +\frac{q}{\ph(q)}\chi(-1)\frac{x^{2\beta}}{\beta^2\log^2 x} +E_{expl}.
\end{equation}

\begin{comm} We note that an alternative approach
via the identity
\[
\frac{q}{\ph(q)} \sum_{\chi(q)} \chi(-1) |\psi(x,\chi)|^2
= \sum_{k=1}^q \Big(\sum_{p\leq x} \Lambda(p) \e\Big(\frac{kp}{q}\Big) \Big)^2
+o(\log^3 x)
\]
and the use of an explicit formula for $\psi(x,\chi)$
might  avoid the use of Gau\ss\ sums in Step 2 of the proof.
\end{comm}

\underline{Step 3.} 
Now we compare the lower bound from Step 1 with the explicit
evaluation from Step 2. With the
assumption $\chi(-1)=-1$ we get the inequality
\[
\frac{x^{2\beta}}{\beta^2\log^2 x}\le
\Big(1-\frac{\delta}{8}\cdot\frac{\ph(q)}{q}\Big)\frac{x^2}{\log^2 x}
+ \Eexpl,
\]
where $\Eexpl$ is the error term of \eqref{explicit} above. 
This yields
\[
x^{2\beta-2}\le \Big(1-\frac{\delta}{8}\cdot\frac{\ph(q)}{q}\Big)
+\Oh\Big(\frac{q}{\ph(q)\log x} 
+ q^3\exp(-\tilde{c}\sqrt{\log x})\Big).
\]

We may now choose $x$  such that 
$(\frac{4\log q}{\tilde{c}} )^2\leq \log x\leq c_{3}\log^{2}q$ for some
$c_{3}>(4/\tilde{c})^{2}$, so that $x$ is
not too large 
compared to $q$, but still such that the choice 
is admissible with our previous assumptions \eqref{eq:qbound}
and \eqref{eq:zwqbound}.
Hence with $1/\log x\leq \tilde{c}^{2}/16\log^{2}q$,
we obtain
\[
x^{2\beta-2}\le \Big(1-\frac{\delta}{8}\cdot\frac{\ph(q)}{q}\Big)
+\frac {c_{1}}{\log^2 q}\frac{q}{\ph(q)} 
\]
for some positive constant $c_{1}$,
since the expression $q/\ph(q)\log x$ dominates the error term due to
\[ 
   q^{3}\exp(-\tilde{c}\sqrt{\log x}) \leq
   q^{3}\exp\Big(-\tilde{c}\frac{4}{\tilde{c}}\log q\Big)
   =q^{3} \exp(-4\log q)= q^{-1}.
\] 
Assume now that $q$ is large enough so that
$\frac{16c_{1}}{\log^2q}\le \delta\frac{\ph^2(q)}{q^2}$. This
means that  
\[x^{2\beta-2}\le
\Big(1-\frac{\delta}{8}\cdot\frac{\ph(q)}{q}\Big)
+\frac {\delta}{16}\frac{\ph(q)}{q} \le 1-\frac
{\delta}{16}\frac{\ph(q)}{q}\]  for $\delta<8$.
This gives the inequality of
Theorem~\ref{thm:bettersiegel}, since
\[
\beta-1\le \frac{\log (1-\frac {\delta}{16}\cdot\frac{\ph(q)}{q})}{2\log x}
\leq 
\frac{-c\ph(q)}{q\log^2 q} \text{ with }c=\frac{\delta}{32c_{3}}>0,
\]
where we use the upper bound for $\log x$, which is 
$\log x \leq c_{3} \log^{2}q$.
\end{proof}

We emphasize that  all constants
in the above proof  are {\em effectively computable} since
the constant $\tilde{c}$ in \eqref{eq:pntap} coming from the prime number theorem in
progressions is itself so and all other constants in the proof
 can be chosen effectively 
depending on $\tilde{c}$.

The next corollary gives a criterion for a composite modulus
$q$ to satisfy Theorem~\ref{thm:bettersiegel}.

\begin{cor} 
\label{cor:cor1}
Assume the WHLE Conjecture
and that $q$ is a sufficiently large integer 
with $\#\{t\mid q\}\cap(\{4\}\cup\{p\equiv
   3 (\md 4);\ p \text{ prime}\})=1$. If there is an
   exceptional zero $\beta$ for a character mod $q$, then
   $1-\beta\ge\frac{c\ph(q)}{q\log^{2}q}$ for some (effective) constant $c>0$.
\end{cor}
\begin{proof}
 For the moduli $q$ in question
there is a single real primitive character such
that $\chi(-1)=-1$. This is certainly true if
$q\in \mathcal{S}:=\{4\}\cup\{p\equiv 3 (\md 4)\}$,
and  $q=tm$ with $t\in \mathcal{S}\cup\{2\}$ and with $m$ having
only prime divisors $p\equiv 1 (\md 4)$. Then there is only a single real
$\chi$ with $\chi(-1)=-1$, namely the one that is induced by that mod
$t$. For the others, $\chi(-1)=1$, since this equation
holds for every prime modulus $p\equiv 1 (\md 4)$.

 Hence, if we assume that $\beta$ exists for an exceptional character
 $\chi$, we know that $\chi$ is real and primitive, and necessarily
 $\chi(-1)=-1$. Then Theorem~\ref{thm:bettersiegel} applies to give
 the assertion.
\end{proof}

We recover Fei's Theorem from Corollary~\ref{cor:cor1}
when $q=p\equiv 3 (\md 4)$.

\medskip
Under GRH except for a possible $\beta>3/4$ and the WHLE Conjecture 
we can deduce the conditional bound
\[
   L(1,\chi)\gg 1-\beta \gg \frac{\ph(q)}{q\log^{2}q}.
\] 
using Theorem~\ref{thm:FI} of Friedlander and Iwaniec, and using
Theorem~\ref{thm:bettersiegel}, supposing $\chi(-1)=-1$ for the
exceptional character $\chi$ mod $q$.
In fact, with Goldfeld's asymptotic formula \eqref{eq:Goldf} 
we can relax the assumption of GRH with one exception 
to obtain

\begin{cor}\label{cor2}
Let the WHLE Conjecture be true. Assuming $L(1,\chi)=o(\log^{-1} q)$, 
we have $L(1,\chi)\gg R_{q}\ph(q)q^{-1}\log^{-2}q$ for an exceptional character 
$\chi$ mod $q$ with $\chi(-1)=-1$.
Here $R_{q}=\sum_{(a,b,c)}a^{-1}$ with the sum going over all reduced 
quadratic forms $(a,b,c)$ of discriminant $-q$.
\end{cor}

A reduced quadratic form of discriminant $-q$ is an integer 
triple $(a,b,c)$ with 
$b^{2}-4ac=-q$ and $-a<b\leq a < \frac{1}{4}\sqrt{q}$, see 
\cite[p.~624]{Gol}.

\begin{proof}
By \eqref{eq:Goldf} from \cite[p.~624]{Gol}
and Theorem~\ref{thm:bettersiegel}, we have
\[
   L(1,\chi)\sim \frac{\pi^{2}}{6}\Big(\sum_{(a,b,c)} a^{-1}\Big)(1-\beta) 
\gg \frac{R_{q}\ph(q)}{q\log^{2}q}.
\]
\end{proof}

Note that for a prime modulus $q=p\equiv 3(\md 4)$, we have
$R_{q}=\sum_{(a,b,c)}a^{-1} \geq 1$ since then there is a reduced quadratic form
$(1,1,c)$ with $a=1$. Then our bound states $L(1,\chi)\gg \log^{-2}q$
under the assumptions of Corollary~\ref{cor2}.

\subsection{Questions}
We would like to know if the case
$\chi(-1)=1$ could be handled as well. 
But we have not been able to combine the results of \cite{BHMS}
together with Fei's approach  which seems to be a natural way
to proceed in this direction.

One could also ask if the WHLE Conjecture itself can be obtained from existing results.
We were not able to find anything appropriate.
Even when averaging the assertion over moduli $q\leq Q$, we only reach
a special case of Conjecture 1 from \cite{KH},
which seems to be out of reach.

%
%

\end{document}